\newcommand{\ba}{\begin{array}}
\newcommand{\ea}{\end{array}}
\newcommand{\bc}{\begin{center}}
\newcommand{\ec}{\end{center}}

\newcommand{\beqn}[1]{\begin{equation}\label{#1}}
\newcommand{\eeqn}{\end{equation}}
\newcommand{\be}{\begin{equation}}
\newcommand{\ee}{\end{equation}}

\newcommand{\beqnn}{\begin{eqnarray}}
\newcommand{\eeqnn}{\end{eqnarray}}

\newtheorem{theorem}{Theorem}
\newtheorem{corollary}{Corollary}
\newtheorem{lemma}{Lemma}
\newtheorem{definition}{Definition}

\newtheorem{problem}{Problem}

\newcommand{\diag}{{\rm diag}}

\documentclass[12pt,draftcls,onecolumn]{IEEEtran}

\ifCLASSINFOpdf
\else
\fi
\usepackage{graphicx}
\usepackage{cite}
\usepackage{remark}
\usepackage{amssymb}  
\newremark{remark}{Remark}
\usepackage[cmex10]{amsmath}

\usepackage{algorithm}
\usepackage{algpseudocode}

\begin{document}
%
\title{Receding Horizon Control Based Consensus Scheme in General Linear Multi-agent Systems}

%
%
%

\author{  Huiping~Li,~\IEEEmembership{member~IEEE,}
          Weisheng~Yan
\thanks{
This work was supported by the start-up research fund of the Northwestern Polytechnical University(NPU);
the basic research foundation of NPU with grant no. GEKY 1004.}
\thanks{The authors are with the Department
of Automation, School of Marine Science and Technology, Northwestern Polytechnical University, Xi'an,
710072, China (e-mail: lihuiping@nwpu.edu.cn;peter.huiping@gmail.com).}
}

%


\maketitle

\begin{abstract}
This paper investigates the consensus problem of general linear multi-agent systems under the framework of optimization.
A novel distributed receding horizon control (RHC) strategy for consensus is proposed.
We show that the consensus protocol generated by the unconstrained distributed RHC can be expressed in an explicit form.
Based on the resulting consensus protocol the necessary and sufficient conditions for ensuring consensus are developed.
Furthermore, we specify more detailed consensus conditions for multi-agent system with general and one-dimensional linear dynamics depending on the difference Riccati equations (DREs), respectively.
Finally, two case studies verify the proposed scheme and the corresponding theoretical results.
\end{abstract}


%
\IEEEpeerreviewmaketitle

\section{INTRODUCTION}\label{sec_introduction}
In last two decades, the cooperative control of networked multi-agent systems has received a lot of
attention due to its wide applications. In particular, the consensus problem is of significant importance, and has inspired much progress, e.g., \cite{Reza_04_consensus_TAC,Moreau_05_TAC_consensus,Wei_05_TAC_consensus}.
In this paper, we are interested in solving the consensus problem of multi-agent systems from the distributed optimal control perspective.
The multi-agent system under study is of fixed directed network topology and general linear time invariant (LTI) dynamics associated with each agent.
The objective of this paper is to design a locally optimal consensus strategy for each agent, and further to investigate under what conditions the closed loop system can achieve consensus by the designed strategy.

The optimality of control protocols brings many desired properties such as phase and gain margin, leading to robustness of the closed loop systems.
The core difficulty of the cooperative optimal control for multi-agent systems lies in the fact that the centralized optimization problem cannot generally be distributed among agents, with few exceptions \cite{Cao_08_TSMC_consensus,Movric_14_TAC_optimal_tracking}.
As a result, the best way of circumventing the difficulty is to utilize the locally optimal control strategy and further combine it with regional information exchange scheme to address the system-level interaction and coupling, approximately achieving some global or cooperative behaviors.

In the literature, one approach to the optimal cooperative control is the linear quadratic regulation (LQR) scheme.
For example, the distributed LQR problem of multi-agent systems with LTI dynamics is studied in \cite{Borrelli_08_TAC_distributedLQR},
showing that the overall stability can be guaranteed by appropriately designing the local LQR and using information exchange among network topology.
The consensus problem with optimal Laplacian matrix for multi-agent systems of first-order dynamics is investigated in \cite{Cao_08_TSMC_consensus}, where it is shown that the design of globally optimal Laplacian matrix can only be achieved by properly choosing the global cost function coupled with the network topology.
Recently, the LQR-based consensus problem of multi-agent systems with LTI dynamics and fixed directed topology is addressed in \cite{Movric_14_TAC_optimal_tracking}, indicating that the globally optimal consensus performance can be achieved by using locally optimal consensus protocol if and only if the overall performance index is selected in a special form depending on the graph structure.

Another way of achieving the (sub-)optimal cooperative control of multi-agent systems is the distributed receding horizon control (RHC) strategy, also known as distributed model predictive control. Based on this approach, there have been many results developed for cooperative stabilization, formation control, and its applications. For example, the distributed RHC-based scheme for cooperative stabilization is proposed in \cite{Dunbar_06_Auto_DMPC,Muller_12_DMPC_IJRNC}, and the formation stabilization is addressed in \cite{Keviczky_06_auto_DMPC} and its application is reported in \cite{Keviczky_08_TCST_formation_MPC}.
Furthermore, the robust distributed RHC problems that can be used for cooperative stabilization are studied in \cite{Richards_IJC_07_RDMPC} for linear systems with coupled constraints and in \cite{Li_TAC_14_RDRHC} for nonlinear systems. To further attack the unreliability of the communication networks, the cooperative stabilization problem of multi-agent nonlinear systems with communication delays are investigated in \cite{Li_SCL_DMPC_14,Franco_08_TAC_DMPC,Li_Auto_14_DRHC}. Note that all of these results use cost functions as Lyapunov functions to prove stability.

Even though it is very desirable to achieve optimal consensus by distributed RHC scheme, there have been few results for the consensus problem of multi-agent systems due to the difficulty that the cost function may not be directly used as Lyapunov function.
In \cite {Ferrari_09_TAC_MPC_consensus}, Ferrari-Trecate {\em et al.} study the consensus problem of multi-agent systems of first-order and second-order dynamics, and the sufficient conditions for achieving consensus are developed by exploiting the geometry properties of the optimal path. Zhan {\em et al.} investigate the consensus problem of first-order sampled-data multi-agent systems in \cite {Zhan_13_auto_consensus}, where state and control input information needs to be exchanged. Note that these two results are only focused on special type of linear systems, which is of limited use.
In \cite{Johansson_08_auto_optimal_consensus} Johansson {\em et al.} propose to use the negotiation to reach the optimal consensus value by implementing the primal decomposition and incremental sub-gradient algorithm, but the effect of the network topology is not explicitly considered.

It can seen that the receding horizon control -based consensus scheme for multi-agent systems with general LTI dynamics has not been solved, and the relationship between consensus and the interplay between the network topology and the RHC design is still unclear, which motivates this study. The main contribution of this paper is two-fold.
\begin{itemize}
\item A novel distributed RHC strategy is proposed for designing the consensus protocol.
      In this strategy, each agent at each time instant only needs to obtain its neighbors's state once via communication network, which is more efficient than the work in \cite{Li_TAC_14_RDRHC,Dunbar_06_Auto_DMPC} (where the state and its predicted trajectory need to be transmitted) and \cite{Muller_12_DMPC_IJRNC,Zhan_13_auto_consensus} (where the neighbors' information needs to be exchanged for many times at each time instant). In addition, we show that the consensus protocol generated by the RHC is a feedback of the linear combination of each agent's state and its neighbors' states, and the feedback gains depend on a set of difference matrix equations. We believe our results partially extend the results in \cite{Ferrari_09_TAC_MPC_consensus,Zhan_13_auto_consensus} to multi-agent systems with LTI dynamics.
\item Given the proposed distributed RHC strategy, a necessary and sufficient condition for ensuring consensus is developed. We show that the consensus can be reached if and only if the network topology contains a spanning tree and a simultaneous stabilization problem can be solved. Furthermore, more specifical sufficient consensus conditions depending on one Reccati difference equation for the multi-agent with LTI dynamics and one-dimensional linear dynamics are also developed, respectively.
\end{itemize}

The remainder of this paper is organized as follows.
Section \ref{Sec_problem_formulation} introduces some well-known results from graph theory and formulates the problem to be studied.
Section \ref{Sec_consensus_protocol} presents the novel distributed RHC scheme, and develops a detailed consensus protocol.
The necessary and sufficient conditions for enuring consensus are proposed in Section \ref{sec_consensus_analysis}, and more specifical sufficient consensus conditions for multi-agent systems with LTI dynamics and one-dimensional linear dynamics are also reported in this section. The case studies are demonstrated in Section \ref{Sec_simulation}.
Finally, the conclusion remarks are summarized in Section~\ref{Sec_Conclusion}.

For the ease of presentation, the following notations are adopted in this paper.
The symbol $\mathbb{R}$ represents the real space.
For a real matrix $A$, its transposition and inverse (if the inverse exists) are denoted as ``$A^{\rm T}$'' and ``$A^{-1}$'', respectively.
If $A$ is a complex matrix, then the transposition is denoted by $A^{\rm H}$.
Given a real (or complex) number $\lambda$, the absolute value (modulus) is defined as $|\lambda|$.
Given a matrix (or a column vector) $X$ and another matrix $P$ with appropriate dimension,
the $2$-induced norm (or the Euclidean norm) of $X$ is denoted by $\|X\|$ and the $P$-weighted norm of $X$
is denoted by $\|X\|_{P} \triangleq \sqrt{X^{\rm T}PX}$.
Given matrix $Q$, $Q>0$ ($Q\geqslant0$) stands for the matrix $Q$ being positive definite (semi-positive definite).
Define the column operation ${\rm col}\{x_1,x_2,\cdots,x_n\}$ as
$[x_1^{\rm T},x_2^{\rm T},\cdots,x_n^{\rm T}]^{\rm T}$, where $x_1, x_2,\cdots,x_n$ are column vectors.
$I_n$ stands for the identity matrix of dimension $n$,
and $\textbf{1}_n$ represents an $n$-dimensional column vector $[1, \cdots, 1]^{\rm T}$.
The symbol $\otimes$ stands for the Kronect product.

\section{Problem Formulation}\label{Sec_problem_formulation}
Consider a multi-agent system of $M$ linear agents. For each agent $i$, the dynamics is described as
\begin{equation}\label{equ_each_agent}
x_i(k+1) = A x_i(k) + B u_i(k),
\end{equation}
where $x_i(k+1) \in \mathbb{R}^{n}$ is the system state,
and $u_i(k) \in \mathbb{R}^m$ is the control input of agent $i$.

There exists a communication network among the $M$-agent system, and the network topology is described as a
directional graph (diagraph) $\mathcal{G} \triangleq \{\mathcal{V},\mathcal{E},\mathcal{A}\}$.
Here, $\mathcal{V} = \{i, i = 1,\cdots,M\}$ is the collection of the nodes of the digraph representing each agent $i$,
$\mathcal{E} \subset \mathcal{V}\times \mathcal{V}$ denotes the edges of paired agents, and $\mathcal{A} =[a_{ij}]\in\mathbb{R}^{M\times M}$ is the adjacency matrix with
$a_{ij}\geqslant 0$. If there is a connection from agent $i$ to $j$, then $a_{ij} = 1$; otherwise, $a_{ij} = 0$. We assume there is no self-circle in the diagraph $\mathcal{G}$, i.e., $a_{ii} = 0$. For each agent $i$, its neighbors are denoted by the agents from which it can obtain information, and the index set for agent $i$'s neighbors is denoted as $\mathcal{N}_i \triangleq \{j|(i,j)\subset \mathcal{E}\}$. The number of agents in $\mathcal{N}_i$ is denoted as $|\mathcal{N}_i|$.
The in-degree of agent $i$ is denoted as $Deg_{in}(i) = \sum_{i=j}^M a_{ij}$, and the degree matrix is denoted as $\mathcal{D}={\rm diag}\{Deg(1),\cdots, Deg (M)\}$. The Laplacian matrix of $\mathcal{G}$ is denoted as $\mathcal{L} = \mathcal{D} - \mathcal{A}$.
Arrange the eigenvalues of $\mathcal{L}$ as $|\lambda_1| \leqslant |\lambda_2| \leqslant \cdots \leqslant |\lambda_M|$.
Assume that the diagraph $\mathcal{G}$ is fixed. Firstly, we recall some standing results from the graph theory \cite{Wei_05_TAC_consensus,Reza_04_consensus_TAC,You_11_TAC_consensus}.

\begin{lemma}\label{lem_spanning_tree}
The digraph $\mathcal{G}$ contains a spanning tree if and only if zero is a simple eigenvalue of the Laplacian matrix $\mathcal{L}$, i.e.,
$0 < |\lambda_2| \leqslant \cdots \leqslant |\lambda_M|$, and the corresponding right eigenvector is ${\rm \textbf{1}}$.
\end{lemma}


For the linear system in (\ref{equ_each_agent}), a standing assumption is made: The pair $[A,B]$ is controllable.
We assume that at each time instant $k$, over the given communication network,
agent $i$ can get state information $x_j(k)$, $j\in\mathcal{N}_i$ from its neighbors.
The communication network is reliable and the information can be transmitted instantaneously without time consumption.

\begin{definition}\cite{You_11_TAC_consensus}
The discrete-time multi-agent system in (\ref{equ_each_agent}) with a given network topology $\mathcal{G}$, and under a distributed control protocol
$u_i(k) = f(x_i(k), x_{-i}(k))$, is said to achieve consensus if,
\begin{equation*}
\lim_{k \rightarrow \infty } \|x_i(k) - x_j(k) \| = 0, j = 1, \cdots, M,
\end{equation*}
where $x_{-i}(k)$ are the collection of agent $i$'s neighbors' states, i.e., $x_{-i}(k) \triangleq \{x_j,j\in\mathcal{N}_i\}$ and $f: \underbrace{\mathbb{R}^n \times \cdots\times \mathbb{R}^n}_{|\mathcal{N}_i| + 1}\rightarrow \mathbb{R}^m$.
\end{definition}
In this paper, we are interested in designing $u_i(k) = f(x_i(k), x_{-i}(k))$ using the distributed RHC strategy for each agent $i$ to achieve consensus.

\section{Distributed RHC based Consensus Protocol}\label{Sec_consensus_protocol}
This section first presents the distributed RHC based consensus strategy, and then develops the detailed consensus protocols for each agent.
\subsection{Distributed RHC Scheme}
For each agent $i$, we propose to utilize the following optimization problem to generate the consensus protocol.
\begin{problem}
\begin{align*}
U_i^*(k) = \arg\min_{\hat{U}_i(k)} J_i(\hat{x}_i(k),\hat{U}_i(k),x_{-i}(k))
\end{align*}
subject to
\begin{align*}
\hat{x}_i(k+n+1|k) = A \hat{x}_i(k+n|k) + B \hat{u}_i(k+n|k),
\end{align*}
where $n = 0, \cdots, N-1$.
Here, $\hat{x}_i(k|k) = x_i(k)$, $\hat{U}_i(k)= \{ \hat{u}_i(k|k), \cdots, \hat{u}_i(k+N-1|k)\}$ and $J_i(\hat{x}_i(k),\hat{U}_i(k),x_{-i}(k)$ is defined as
\begin{align}
  & J_i(\hat{x}_i(k),\hat{U}_i(k),x_{-i}(k))\nonumber\\
= &\sum_{n=0}^{N-1} \left(\|\hat{u}_i(k+n|k)\|_{R_i}^2 + \sum_{j\in \mathcal{N}_i}a_{ij}\|\hat{x}_i(k+n|k) - x_j(k)\|_{Q_i}^2\right)\nonumber\\
  & + \sum_{j\in \mathcal{N}_i}a_{ij}\|\hat{x}_i(k+N|k) - x_j(k)\|_{Q_{iN}}^2,
\end{align}
where $R_i>0$, $Q_i>0$ and $Q_{iN}>0$ are symmetric matrices, and $N>0$ is a positive integer which is called the prediction horizon.
\end{problem}

At each time instant $k$, each agent gets its neighbors' state information $x_j(k)$, $j\in\mathcal{N}_i$, solves Problem 1 and then uses $u^*_i(k|k)$ as the control input to achieve consensus.
\begin{remark}
In Problem 1, we use the term $\sum_{j\in \mathcal{N}_i}a_{ij}\|\hat{x}_i(k+n|k) - x_j(k)\|_{Q_i}^2$ to achieve consensus. Note that the consensus term can be rewritten as $|\mathcal{N}_i|\|\hat{x}_i(k+n|k) - ave(x_{-i}(k))\|_{Q_i}^2$, where $ave(x_{-i}(k))$ is the average of agent $i$'s neighboring states, which is consistent with the consensus term in \cite{Ferrari_09_TAC_MPC_consensus}. However, Problem 1 generates the first-order and second-order cases in \cite{Ferrari_09_TAC_MPC_consensus} to general linear systems. In addition, the solve of Problem 1 only requires the exchange of agent $i$ neighbors' state for one time, which results in less communication load in comparison with these in \cite{Muller_12_DMPC_IJRNC,Zhan_13_auto_consensus}, where more information is needed to exchange for consensus.
\end{remark}

\subsection{Specific Consensus Protocol}
This subsection shows that the consensus protocol generated by solving Problem 1 is can be expressed in an explicit form.
By using the convex optimization approach, the result is reported in the following theorem.
\begin{theorem}\label{thm_consensus_law}
For the system in (\ref{equ_each_agent}) with network topology $\mathcal{G}$ and the optimization problem 1, the optimal solution is given as:
\begin{equation*}
\hat{u}_i^*(k+n|k) = \sum_j^{M} a_{ij} [K_i(n) \hat{x}_i(k+n|k) + G_i(n) x_j(k)],
\end{equation*}
Here, $n = 0,\cdots,N-1$, $K_i(n) = -R^{-1}_i B^{\rm T} [I + \sum_{j=1}^{M}a_{ij} P_i(n+1)B R^{-1}_iB^{\rm T}]^{-1} P_i(n+1) A$, and
$G_i(n) = -R^{-1}_i B^{\rm T} [I + \sum_{j=1}^{M}a_{ij} P_i(n+1)B R^{-1}_i B^{\rm T}]^{-1} \Delta_i(n+1)$,
where $P_i(n+1)$ and $\Delta_i(n+1)$ satisfy the following matrix equations, respectively:
\begin{align}
P_i(n) = & A^{\rm T}[I + \sum_{j=1}^{M}a_{ij} P_i(n+1)B R^{-1}_i B^{\rm T}]^{-1}P_i(n+1)A + Q_i,\label{equ_Pn} \\
\Delta_i(n) =& A^{\rm T}[I + \sum_{j=1}^{M}a_{ij} P_i(n+1)B R^{-1}_i B^{\rm T}]^{-1}\Delta_i(n+1) -Q_i,\label{equ_Deltan}
\end{align}
with the initial conditions $P_i(N) = Q_{iN}$ and $\Delta_i(N) = -Q_{iN}$.
Furthermore, the consensus protocol for each agent $i$ is given as:
\begin{equation}\label{equ_consensus_law}
u_i(k) = u_i^*(k|k).
\end{equation}
\end{theorem}

\begin{proof}
By introducing the lagrange multipliers $\lambda_i(n+1) \in \mathbb{R}^n, n = 0, \cdots, N-1$, construct a lagrange function as
\begin{align*}
L(\hat{U}_i(k),\hat{X}_i(k)) \triangleq & \frac{1}{2} J_i(\hat{x}_i(k),\hat{U}_i(k),x_{-i}(k)) \\
                               & + \sum_{n=0}^{N-1}\lambda^{\rm T}_i(n+1)[A\hat{x}_i(k+n|k)\nonumber \\
                               &  + B \hat{u}_i(k+n|k) - \hat{x}_i(k+n+1|k)],
\end{align*}
where $\hat{X}_i(k) \triangleq \{\hat{x}_i(k|k),\cdots,\hat{x}_i(k+N-1|k)\}$.
Since $Q_i>0, Q_{iN}>0$ and $R_i>0$, the objective function is strictly convex. According to the the Karush-Kuhn-Tucker (KKT) conditions \cite{Lu_12_optimization}, there exists a unique global minima for Problem 1, which satisfies the following conditions:
\begin{align}
\frac{\partial L}{\partial \hat{x}_i^*(k+n|k)} = & \sum_{j=1}^{M} Q_i a_{ij}[\hat{x}_i^*(k+n|k)-x_j(k)]+ A^{\rm T} \lambda_i(n+1) - \lambda(n) = 0,\label{equ_KKT1} \\
\frac{\partial L}{\partial \hat{x}_i^*(k+N|k)} = & \sum_{j=1}^{M} Q_{iN} a_{ij}[\hat{x}_i^*(k+N|k)-x_j(k)]- \lambda_i(N) = 0,\label{equ_KKT2}\\
\frac{\partial L}{\partial \hat{u}_i^*(k+n|k)} & = R_i \hat{u}_i^*(k+n|k) + B^{\rm T} \lambda_i(n+1) = 0,\label{equ_KKT3}\\
A\hat{x}_i^*(k+n|k)& + B \hat{u}_i^*(k+n|k) - \hat{x}_i^*(k+n+1|k) = 0,\label{equ_KKT4}
\end{align}
where $n = 0, \cdots, N-1$.

In what follows, we first show an equation to evaluate $\lambda_i(n)$, by mathematical induction as follows:
\begin{equation}\label{equ_lamdai}
\lambda_i(n) = \sum_{j=1}^{M} [P_i(n) \hat{x}_i^*(k+n|k) + \Delta_i(n) x_j(k)],
\end{equation}
where, $ n = N, \cdots, 0$.
Using (\ref{equ_KKT3}), we can get
\begin{equation}\label{equ_u_k}
\hat{u}_i^*(k+n|k) = -R_i^{-1} B^{\rm T} \lambda_i(n+1).
\end{equation}
When $n=N$, from (\ref{equ_KKT2}), we can obtain that
\begin{align*}
\lambda_i(N) = & \sum_{j=1}^{M} a_{ij}Q_{iN}[\hat{x}_i^*(k+N|k)-x_j(k)]\nonumber\\
             = & \sum_{j=1}^{M} a_{ij}[P_i(N)\hat{x}_i^*(k+N|k) + \Delta_i(N) x_j(k)].
\end{align*}
Assume that (\ref{equ_lamdai}) holds for some $n = l+1$ with $ l \leqslant N-1$.
When $n = l$, by combining (\ref{equ_KKT1}), (\ref{equ_KKT4}) and (\ref{equ_u_k}), we have
\begin{align*}
\lambda_i(l) = & A^{\rm T} [I + \sum_{j=1}^{M}a_{ij} P_i(n+1)B R^{-1}_i B^{\rm T}]^{-1}\nonumber\\
               & \times [\sum_{j=1}^{\rm M} a_{ij} (P_i(l+1) A\hat{x}_i^*(k+n|k)+ \Delta_i(l+1) x_j(k))]\nonumber\\
               &  + \sum_{j=1}^{M} a_{ij} Q_i(\hat{x}_i^*(k+n|k) - x_j(k))\nonumber\\
=& \sum_{j=1}^{M} a_{ij}[P_i(l) \hat{x}_i(k+n|k) + \Delta_i(l) x_j(k)].
\end{align*}
Therefore, using the mathematical induction, (\ref{equ_lamdai}) has been proven.
Finally, the result in Theorem \ref{thm_consensus_law} follows by plugging (\ref{equ_lamdai}) into (\ref{equ_u_k}).
This completes the proof.
\end{proof}

Note that $P_i(n)$ satisfies a modified Reccati difference equation in (\ref{equ_Pn}), which depends on the network topology.
In order to decouple (\ref{equ_Pn}) from the network topology, we design $R_i$ such that
\begin{equation}\label{equ_design_R}
R_1/|\mathcal{N}_1| = \cdots = R_M/|\mathcal{N}_M| = R,
\end{equation}
where $R>0$. Using (\ref{equ_design_R}), (\ref{equ_Pn}) and (\ref{equ_Deltan}) become
\begin{align}
P_i(n) = A^{\rm T}[I + P_i(n+1)B R^{-1} B^{\rm T}]^{-1}P_i(n+1)A + Q_i,\label{equ_RDE_Pn} \\
\Delta_i(n) = A^{\rm T}[I + P_i(n+1)B R^{-1} B^{\rm T}]^{-1}\Delta_i(n+1) -Q_i.\label{equ_RDE_Deltan}
\end{align}
Therefore, $P_i(n)$ follows a standard RDE in (\ref{equ_RDE_Pn}).

\section{Consensus Analysis}\label{sec_consensus_analysis}
This sections first develops a necessary and sufficient condition for achieving consensus of general linear agent systems, then further propose the sufficient conditions for multi-agent system of general LTI and one-dimensional dynamics, respectively.
\subsection{Necessary and Sufficient Condition}
The necessary and sufficient condition for reaching consensus is as follows.
\begin{theorem}\label{thm_iff}
For the system in (\ref{equ_each_agent}) with the network topology $\mathcal{G}$ and the consensus protocol in (\ref{equ_consensus_law}), assume that $R_i$ is designed as in (\ref{equ_design_R}). Then the consensus can be reached if and only if (a) $\mathcal{G}$ contains a spanning tree, and (b) $A-BK_i + (\lambda_i - 1)B G_i$ is stable for all $\lambda_i$, $i=2,\cdots, M$, where $\lambda_i$ are the nonzero eigenvalues of $\Gamma$ defined as $\Gamma = {\rm diag}(1/|\mathcal{N}_1|, \cdots, 1/|\mathcal{N}_M|) \mathcal{L}$, and
\begin{align*}
K_i \triangleq &  -R^{-1}B^{\rm T} [I + P_i(1) B R^{-1} B^{\rm T}]^{-1} P_i(1) A\\
G_i \triangleq &  -R^{-1}B^{\rm T}[I + P_i(1) B R^{-1} B^{\rm T}]^{-1} \Delta_i(1).
\end{align*}
\end{theorem}

Before proving Theorem \ref{thm_iff}, a lemma is first needed.
\begin{lemma}\label{lemma_new_graph}
If the graph $\mathcal{G}$ contains a spanning tree, then a new graph $\tilde{\mathcal{G}} \triangleq \{\mathcal{V}, \mathcal{E}, \tilde{\mathcal{A}}\}$ also contains a panning tree, with $\sum_{j=1}^{M} \tilde{a}_{ij} = 1$ for all $i$, and the corresponding Laplacian matrix is $\Gamma$, where $\tilde{\mathcal{A}} = {\rm diag}(1/|\mathcal{N}_1|, \cdots, 1/|\mathcal{N}_M|) \mathcal{A}$.
\end{lemma}

The proof of lemma \ref{lemma_new_graph} is straightforward, so it is omitted here.

\begin{proof}
For each agent $i$, plugging the consensus law in (\ref{equ_consensus_law}) into (\ref{equ_each_agent}), and doing some algebraic operations, we have
\begin{align*}
x_i(k+1) = A x_i(k) - \sum_{j=1}^{M}\tilde{a}_{ij}(B K_i x_i(k) + B G_i x_j(k)),
\end{align*}
where $\tilde{a}_{ij}$ is the element in $\tilde{\mathcal{A}}$.
Define $\delta_i(k) = x_i(k) - x_1(k)$, and we have
\begin{align*}
\delta_i(k) = (A-BK_i) \delta_i(k) - \sum_{j=1}^{M}(\tilde{a}_{ij}-\tilde{a}_{1j}) B G_i \delta_j(k),
\end{align*}
where $i = 2, \cdots, M$.
By writing the above equations in an augmented form, one gets
\begin{align*}
& \delta(k+1) = (I_{M-1}\otimes(A-BK_i))\delta(k)\nonumber\\
&  - \left[\left(
                                                    \begin{array}{ccc}
                                                      \tilde{a}_{22}- \tilde{a}_{12} & \ldots & \tilde{a}_{2M}- \tilde{a}_{1M}\\
                                                      \vdots & \ddots & \vdots \\
                                                      \tilde{a}_{M2}- \tilde{a}_{12} & \ldots & \tilde{a}_{MM}- \tilde{a}_{1M} \\
                                                    \end{array}
                                                  \right)
\otimes (BG_i)\right]\delta(k),
\end{align*}
where $\delta(k) = {\rm col}(\delta_2(k), \cdots, \delta_M(k))$. The above equation can be further written as
\begin{align}\label{equ_closed_A_delta}
\delta(k+1) = A_\delta \delta(k)
\end{align}
where $A_\delta = (I_{M-1}\otimes(A-BK_i))+ (\tilde{\mathcal{L}}_p + \textbf{1}_{M-1} \tilde{a}_1^{\rm T} \otimes B G_i) - (I_{M-1}\otimes B G_i)$, and $ \tilde{\mathcal{L}}_p = \left[
                                                    \begin{array}{cccc}
                                                      1 & -\tilde{a}_{22} &  \ldots & \tilde{a}_{2M}\\
                                                      -\tilde{a}_{32} & 1 & \ldots &  -\tilde{a}_{3M}\\
                                                      \vdots & \vdots & \ddots & \vdots \\
                                                      -\tilde{a}_{M2} & -\tilde{a}_{M3} & \ldots & 1\\
                                                    \end{array}
                                                  \right]$ and $\tilde{a}_1 = [\tilde{a}_{12}, \cdots, \tilde{a}_{1M}]^{\rm T}$.
From (\ref{equ_closed_A_delta}), it can be seen that the consensus can be reached if and only if all the eigenvalues of $A_\delta$ are in the unit circle. Next, we need to analyze the properties of the eigenvalues of $A_\delta$.
According to Lemma \ref{lemma_new_graph} and Lemma \ref{lem_spanning_tree}, $\Gamma$ has exactly one zero eigenvalue. Put the nonzero eigenvalues of $\Gamma$ in order as $|\lambda_2|\leqslant \cdots \leqslant |\lambda_M|$. Since $S^{-1} \Gamma S = \left[
                           \begin{array}{cc}
                             0 & -\tilde{a}_1^{\rm} \\
                             0 & \tilde{\mathcal{L}}_p + \textbf{1}_{M-1} \tilde{a}_1^{\rm T} \\
                           \end{array}
                         \right]
$, the eigenvalues of $\tilde{\mathcal{L}}_p + \textbf{1}_{M-1} \tilde{a}_1^{\rm T}$ are $\lambda_2, \cdots, \lambda_M$.
Therefore, there exists a nonsingular matrix $T$, such that
\begin{align*}
T^{-1}(\tilde{\mathcal{L}}_p + \textbf{1}_{M-1} \tilde{a}_1^{\rm T} ) T = J = {\rm \diag}(J_1, \cdots, J_s),
\end{align*}
where $J_k$, $k=1,\cdots, s$ are upper triangular Jordan blocks and the principle diagonal elements are $\lambda_2, \cdots, \lambda_M$.
As a result, we have
\begin{align*}
& (T\otimes I_n)^{-1} A_\delta (T\otimes I_n) \nonumber \\
= & I_{M-1}\otimes(A-BK_i) + J\otimes B G_i - I_{M-1} B G_i.
\end{align*}
Thus, the eigenvalues of $A_\delta$ are $A-BK_i + (\lambda_i -1) B G_i$, for all $i = 2, \cdots, M$.
That is, the consensus can be reached if and only if $A-BK_i + (\lambda_i -1) B G_i$ are stable.
The proof is completed.
\end{proof}

Note that $K_i$ and $G_i$ depends on the matrix equations in (\ref{equ_RDE_Pn}) and (\ref{equ_RDE_Deltan}), respectively.
In particular, the (\ref{equ_RDE_Deltan}) is not a RDE, and thus might be of complex properties. In the following, we develop more detailed sufficient conditions to facilitate design.
\subsection{Sufficient Conditions for General Linear Systems}
In this section, we develop sufficient conditions for ensuring consensus, which are reported as follows.
\begin{theorem}\label{col_sufficient_LTI}
For the system in (\ref{equ_each_agent}) with the network topology ${\mathcal{G}}$, assume that ${\mathcal{G}}$ contains a spanning tree, and $R_i$ is designed as in (\ref{equ_design_R}). If $Q_{iN}$ and $Q_i$ are designed such that
\begin{align}
&  Q_{iN} - P_i(N-1) >0, \label{equ_monotically_P}\\
& Q_i - \Xi_i  -\|\lambda_j-1\|^2\nonumber\\
&\times \|I + B(R+B^{\rm T}PB)^{-1}B^{\rm T}\|_{[\bar{\rho}^{N-1}_i Q_{iN} + \sum_{l=1}^{N-2}\bar{\rho}_i^l Q_i]}^2 > 0,\label{equ_sufficient_Q}
\end{align}
where $\Xi_i \triangleq \| B(R + B^{\rm T} P_i(1) B)^{-1}B^{\rm T} P_i(1) \bar{A}_i\|^2$, $\bar{A}_i = A - BK_i$, $\bar{\rho}_i = \rho(A^{\rm T}(I + P_i(1) B R^{-1}B)^{-1})$ and $\lambda_j$, $j=2, \cdots, M$, are the eigenvalues of $\tilde{\mathcal{L}}$.
\end{theorem}

Before developing the proof of Theorem \ref{col_sufficient_LTI}, three lemmas are needed.
\begin{lemma}\label{lemma_RDE}
For the RDE in (\ref{equ_RDE_Pn}), if $Q_{iN} - P_i(N-1)>0$, then $ P_i(N)>\cdots, >P_i(1)$, and $\hat{u}_i^*(k+n|k) = BK_i(k+n|k)$ is a stabilizing controller law for the system in (\ref{equ_each_agent}), $n = N-1, \cdots, 0$.
\end{lemma}
The proof can be directly followed using Theorem 1 and Lemma 1 in \cite{Poubelle_88_TAC_FakeReccati} by considering $Q_i>0$.
\begin{lemma}\label{lemma_special_G_i}
For the RDE in (\ref{equ_RDE_Pn}) and all the eigenvalues $\lambda_2, \cdots, \lambda_M$ of $\tilde{\mathcal{L}}$, the following holds:
\begin{align}\label{equ_temp_Pi(1)}
& P_i(1) -(A- \lambda_i B K_i)^{\rm H} P_i(1)(A- \lambda_i B K_i) \nonumber\\
> & (1-\bar{\lambda}_i)P_i(1) + (\bar{\lambda}_i - 1)A^{\rm T}P_i(1)A + \bar{\lambda} Q_i,
\end{align}
where $\bar{\lambda}_i = \lambda_i^{\rm H} + \lambda_i - \|\lambda_i\|^2$.
\end{lemma}
\begin{proof}
According to \cite{Anderson_71_LOC}, $K_i$ can also been written as
\begin{equation}\label{equ_K_i_newform}
K_i = (R + B^{\rm T} P_i(1) B)^{-1}B^{\rm T} P_i(1) A.
\end{equation}
Using (\ref{equ_K_i_newform}), we have
\begin{align*}
  & P_i(1) - (A- \lambda_i B K_i)^{\rm H} P_i(1)(A- \lambda_i B K_i) \nonumber\\
= & P_i(1) - A^{\rm T} P_i(1)A + (\lambda_i^{\rm H} + \lambda_i) A^{\rm T} P_i(1)\nonumber\\
  &\times  B (R+ B^{\rm T}P_i(1)B)^{-1}B^{\rm T} P_i(1)A \nonumber\\
& - \|\lambda_i\|^2  A^{\rm T} P_i(1) B (R+ B^{\rm T}P_i(1)B)^{-1}\nonumber \\
&\times B^{\rm T} P_i(1) B (R + B^{\rm T} P_i(1) B)^{-1} B^{\rm T} P_i(1) A\nonumber\\
\geqslant & P_i(1) - A^{\rm T} P_i(1)A \nonumber\\
& + \bar{\lambda_i} A^{\rm T} P_i(1) B (R+ B^{\rm T}P_i(1)B)^{-1}B^{\rm T} P_i(1)A.
\end{align*}
The DRE in (\ref{equ_RDE_Pn}) can also be written as
\begin{align*}
P_i(n) = & Q_i + A^{\rm T} P_i(n+1) A - A^{\rm T} P_i(n+1) B\nonumber\\
&\times (R+ B^{\rm T}P_i(n+1)B)^{-1}B^{\rm T} P_i(n+1)A.
\end{align*}
Therefore, we further have
 \begin{align*}
  & P_i(1) - (A- \lambda_i B K_i)^{\rm H} P_i(1)(A- \lambda_i B K_i) \nonumber\\
\geqslant & P_i(1) - A^{\rm T} P_i(1)A + \bar{\lambda_i} (Q_i + A^{\rm T} P_i(1) - P_i(0)).
\end{align*}
According to Lemma \ref{lemma_RDE}, it can be derived that $P_i(0) < P_i(1)$.
As a result, (\ref{equ_temp_Pi(1)}) follows. This completes the proof.
\end{proof}
The following lemma gives a bound of the sequence $\{\Delta_i(n)\}$.
\begin{lemma}\label{lemma_delta_auxiliary}
For any positive definite matrix $\Pi\in \mathbb{R}^{n\times n}$, the following holds:
\begin{equation*}
\|\Pi\|^2_{\Delta_i(n)} \leqslant \|\Pi\|^2_{\bar{\rho}^{N-n}_i Q_{iN} + \sum_{l=1}^{N-(n+1)}\bar{\rho}_i^l Q_i},
\end{equation*}
where $n = N, \cdots, 0$.
\end{lemma}
\begin{proof}
Construct an auxiliary series $\{\tilde{\Delta}_i(n)\}$ as follows:
\begin{equation*}
\tilde{\Delta}_i(n) = \bar{\rho}_i \tilde{\Delta}_i(n+1) + Q_i,
\end{equation*}
where $n = N-1, \cdots, 0$, and $\tilde{\Delta}_i(N) = Q_{iN}$.
Using Lemma \ref{lemma_RDE}, it can be seen that $\tilde{\Delta}_i(n) + \Delta_i(n) \geqslant 0$.
On the other hand, we have
\begin{equation*}
\tilde{\Delta}_i(n) = \bar{\rho}^{N-n}_i Q_{iN} + \sum_{l=1}^{N-(n+1)}\bar{\rho}_i^l Q_i.
\end{equation*}
Thus, the result in Lemma \ref{lemma_delta_auxiliary} holds.
\end{proof}

\emph{Proof of Theorem \ref{col_sufficient_LTI}}: By direct calculation, we have
\begin{align*}
& P_i(1) - [A-BK_i + (\lambda_i - 1) B G_i]^{H} \nonumber \\
& \times P_i(1)[A-BK_i + (\lambda_i -1)B G_i]\nonumber\\
= & P_i(1) - \bar{A}_i^{\rm T} P_i(1) \bar{A}_i -(\lambda_i-1)^{\rm H} \Delta_i(1)^{\rm T}\nonumber \\
& \times B(R + B^{\rm T}P_i(1)B)^{-1} B^{\rm T} P_i(1)\bar{A}_i\nonumber\\
 & -(\lambda_i -1) \bar{A}_i^{\rm T}P_i(1)B(R + B^{\rm T}P_i(1)B)^{-1}B^{\rm T}\Delta_i(1)\nonumber\\
 & -\|\lambda_i -1\|^2 \Delta_i(1)^{\rm T} B (R + B^{\rm T}P_i(1)B)^{-1}\nonumber \\
 & \times B^{\rm T}P_i(1)B(R + B^{\rm T}P_i(1)B)^{-1} B^{\rm T} \Delta_i(1),
\end{align*}
where $G_i = (R+B^{\rm T}P_i(1)B)^{-1}B^{\rm T}\Delta_i(1)$ has been used.
According to Lemma 5, we further have
\begin{align*}
& P_i(1) - [A-BK_i + (\lambda_i - 1) B G_i]^{H} \nonumber \\
& \times P_i(1)[A-BK_i + (\lambda_i -1)B G_i]\nonumber\\
\geqslant & P_i(1) - \bar{A}_i^{\rm T} P_i(1) \bar{A}_i - \|\lambda_i -1\|^2 \Delta_i(1)^{\rm T}\nonumber \\
& \times (I + B(R+B^{\rm T}P_i(1)B)^{-1}B^{\rm T}) \Delta_i(1)- \Pi_i.
\end{align*}
Using Lemma \ref{lemma_special_G_i} with $\lambda_i = 1$, we get
\begin{align*}
& P_i(1) - [A-BK_i + (\lambda_i - 1) B G_i]^{H} \nonumber \\
& \times P_i(1)[A-BK_i + (\lambda_i -1)B G_i]\nonumber\\
> & Q_i + \|\lambda_i -1\|^2 \Delta_i(1)^{\rm T}\nonumber \\
& \times (I + B(R+B^{\rm T}P_i(1)B)^{-1}B^{\rm T}) \Delta_i(1)- \Pi_i.
\end{align*}
Finally, applying the result in Lemma \ref{lemma_delta_auxiliary} and using the condition in (\ref{equ_sufficient_Q}), we obtain that
$P_i(1) - [A-BK_i + (\lambda_i - 1) B G_i]^{H} P_i(1)[A-BK_i + (\lambda_i -1)B G_i]>0$.
That is, $A-BK_i + (\lambda_i - 1) B G_i$ is stable for all $\lambda_2,\cdots, \lambda_M$.
According to Theorem \ref{thm_iff}, the consensus can be reached. This completes the proof.
\subsection{Sufficient Conditions for One-dimensional Systems}
For each agent $i$ of one-dimensional linear dynamics, the system equation becomes
\begin{equation}\label{equ_one_dimension_system}
x_i(k+1) = a x_i(k) + b u_i (k).
\end{equation}
For one-dimensional systems, the corresponding parameters $R, Q_i, Q_{iN}, P_i(n), \Delta_i(n)$ reduce to scalars $r, q_i, q_i^N, p_i(n),\delta_i(n)$.
The result for ensuring consensus is reported in the following corollary.
\begin{corollary}\label{cor_consensus_1D}
For system in (\ref{equ_one_dimension_system}) with the network topology $\mathcal{G}$, assume that $r_i$ is designed as in (\ref{equ_design_R}) and the  $\mathcal{G}$ contains a panning tree.
If $q_i$, $q_i^N$ and $r$ are designed such that
\begin{align}
& p_i(N) - p_i(N-1) >0,\label{equ_qi_qiN}\\
& |\alpha_i|^{N-1} q_i^N + \sum_{l = 1}^{N-2}|\alpha_i|^l q_i < \bar{\theta}_{min},\label{equ_qn_bound}
\end{align}
then all the states of each agent can reach consensus. Here, $\alpha_i \triangleq \frac{r a}{r + b^2 p_i(1)}$,
$\bar{\theta}_{min} = \min_{i}\{\theta_i\}$, $\theta_i = |ar|\frac{\sqrt{(1-a_{ci}^2)b_i^2 + (1-a_i)^2 a_{ci}^2}-(|1-a_i|a_c^2)}{(1-a_i)^2a_{ci}^2 + b_i^2}$,
$a_{ci} = \frac{ar}{r + b^2 p_i(1)}$, $a_i = {\rm Re}(\lambda_i)$, $b_i = {\rm Imag}(\lambda_i)$, and $\lambda_i$ $i=2, \cdots, M$ are the eigenvalues of $\tilde{\mathcal{L}}$.
\end{corollary}
\begin{proof}
By plugging $K_i = \frac{ab p_i(1)}{r + b^2 p_i(1)}$ and $G_i = \frac{b\delta_i(1)}{r + b^2 p_i(1)}$ into $A-B K_i + (\lambda_i -1) B G_i$, we have
\begin{align*}
A-BK_i + (\lambda_i -1) B G_i  = \frac{ar + (1-\lambda_i)b^2 \delta_i(1)}{r + b^2 p_i(1)}.
\end{align*}
To achieve consensus, we need to develop conditions to ensure that $|\frac{ar + (1-\lambda_i)b^2 \delta_i(1)}{r + b^2 p_i(1)}|<1$ for $\lambda_2, \cdots, \lambda_M$. To that need, we consider the term $T\triangleq |a r + (1-\lambda_i)b^2 \delta_i(1)|^2-|r + b^2 p_i(1)|^2$.
Firstly, we derive an upper bound of $\delta_i(1)$. Define a sequence $\{\tilde{\delta}_i(n), n = N-1, \cdots, 0\}$ as
\begin{equation}\label{equ_tilde_delta}
\tilde{\delta}_i(n) = \frac{ar\tilde{\delta}_i(n+1)}{r + b^2 p_i(1)} + q_i,
\end{equation}
where $\tilde{\delta}_i(N) = q_{iN}$. Because of (\ref{equ_qi_qiN}), Lemma \ref{lemma_RDE} can be used, implying $p_i(1)< \cdots < p_i(N)$. As a result, $|\tilde{\delta}_i(n)| \geqslant |\delta_i(n)|$.
Therefore, a sufficient condition for guaranteeing $T<0$ is
\begin{align}\label{equ_1D_inequality}
(|ar| + |1-a_i| b^2 |\tilde{\delta}_i(1)|)^2 + b_i^2 b^2 |\tilde{\delta}_i(1)|^2 -(r + b^2 p_i(1))^2 < 0.
\end{align}
By some algebraic operations, (\ref{equ_1D_inequality}) reduces to
\begin{align}\label{equ_1D_inequality_second}
(|c_i|^2 + |d_i|^2) (\tilde{\delta}_i(1))^2 + 2 |a_c| |c_i| |\tilde{\delta}_i(1)| + |a_c|^2 - 1 <0,
\end{align}
where $c_i = \frac{|1-a_i|b^2 |a_c|}{ar}$ and $d_i = \frac{b^2 b_i}{ar}$.

On the other hand, from (\ref{equ_tilde_delta}), we can obtain that $|\tilde{\delta}_i(1)|\leqslant |\alpha_i|^{N-1} q_{iN} + \sum_{l = 1}^{N-2}|\alpha_i|^l q_i$.
Using the condition in (\ref{equ_qn_bound}), we have $|\tilde{\delta}_i(1)|\leqslant \bar{\theta}_{min}$, which ensures (\ref{equ_1D_inequality_second}) holds.
This completes the proof.
\end{proof}

\section{Simulation Studies}\label{Sec_simulation}
In this section, two examples are provided to verify the theoretical results.
\subsection{One-dimensional Case}
Consider a multi-agent system with $5$ agents, and the model for each agent is given as
\begin{equation*}
x_i(k+1) = 2 x_i(k) + u_i(k).
\end{equation*}
The adjacency matrix $\mathcal{A}$ of $\mathcal{G}$ is given as
$\mathcal{A} = \left[
   \begin{array}{ccccc}
     0 & 1 & 0 & 0 & 1 \\
     0 & 0 & 1 & 0 & 0 \\
     0 & 0 & 0 & 1 & 0 \\
     1 & 0 & 1 & 0 & 0 \\
     1 & 1 & 1 & 0 & 0 \\
   \end{array}
 \right]
$. The Laplacian matrix $\tilde{\mathcal{L}}$ of $\tilde{\mathcal{G}}$ can be figured out as $\tilde{\mathcal{L}} = \left[
   \begin{array}{ccccc}
     1 & -\frac{1}{2} & 0 & 0 & -\frac{1}{2} \\
     0 & 1 & -1 & 0 & 0 \\
     0 & 0 & 1 & -1 & 0 \\
     -\frac{1}{2} & 0 & -\frac{1}{2} & 1 & 0 \\
     -\frac{1}{3} & -\frac{1}{3}  & -\frac{1}{3}  & 0 & 1 \\
   \end{array}
 \right]$.
It can be seen that $\mathcal{G}$ contains a panning tree.
The parameters $q_i$, $q_{iN}$ are designed as $q_i=2$, and $q_{iN}=6$, $i=1,\cdots, 5$.
The parameter $r$ is designed as $r = 1$, and thus, the corresponding $r_i$, $i=1,\cdots, 5$,
are $r_1 = \frac{1}{2}$, $r_2 = 1$, $r_3 = 1$, $r_4 = \frac{1}{2}$, and $r_5 = \frac{1}{3}$.
The horizon is chosen as $N=3$.
By checking the condition in (\ref{equ_qi_qiN}), we have $p_i(3)-p_i(2) = 0.5714>0$.
Furthermore, $\bar{\theta}_{min}$ is calculated as $3.3391$,
and the term $|\alpha_i|^{2} q_{iN} + |\alpha_i| q_i = 1.2186$.
Thus, (\ref{equ_qn_bound}) verifies. According to Corollary \ref{cor_consensus_1D}, the multi-agent system will achieve consensus.
We plot the simulated trajectories of the five agents in Fig. \ref{fig_1d_state}, showing that the system can reach consensus.
\begin{figure}[!hbt] \centering
\includegraphics[width=0.5\textwidth]{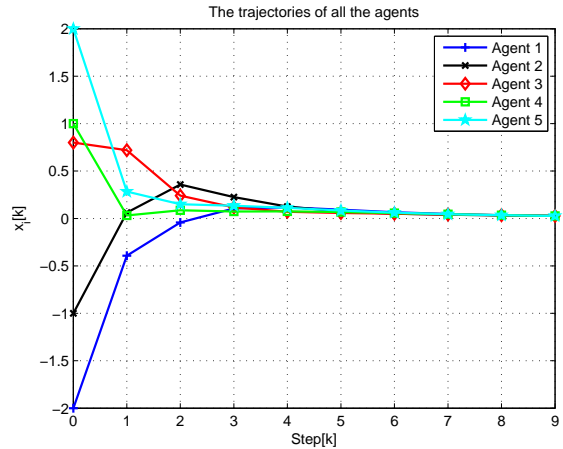}
\caption{Trajectories for 5 agents with 1D system dynamics.}\label{fig_1d_state}
\end{figure}

\subsection{General Linear Systems Case}
Consider a multi-agent system with $3$ agents, and each of a general LTI dynamics.
The system matrices are $A = \left[
                               \begin{array}{cc}
                                 2 & 0 \\
                                 1.2 & -1 \\
                               \end{array}
                             \right]$,
$B = \left[\begin{array}{c}
1 \\
1 \\
\end{array}
\right]$.
The diagraph $\mathcal{G}$ contains a spanning tree, and it adjacency matrix $\mathcal{A}$ is $\mathcal{A} = \left[
                                                                        \begin{array}{ccc}
                                                                          0 & 1 & 1 \\
                                                                          0 & 0 & 1 \\
                                                                          1 & 1 & 0 \\
                                                                        \end{array}
                                                                      \right]$.
Thus, the corresponding matrix $\tilde{\mathcal{L}} = \left[
                                                      \begin{array}{ccc}
                                                        1 & -\frac{1}{2} & -\frac{1}{2} \\
                                                        0 & 1 & -1 \\
                                                        -\frac{1}{2} & -\frac{1}{2} & 1 \\
                                                      \end{array}
                                                    \right]
$.
The matrices $Q_i$ and $Q_{iN}$ are designed as $Q_i = {\rm \diag}(2, 2)$ and $Q_{iN} = {\rm \diag}(15, 20)$, $i=1, 2, 3$.
The matrices $R_i$ are designed as $R_1 = \frac{1}{2}$, $R_2 = 1$ and $R_3 = \frac{1}{2}$, respectively, which satisfies (\ref{equ_design_R}) with $R = 1$. The prediction horizon is chosen as $N = 10$.

Calculate $P_i(10) - P_i(9) = \left[
                                \begin{array}{cc}
                                  5.2  & -6 \\
                                  -6 & 9.111 \\
                                \end{array}
                              \right]>0
$. Thus, the condition in (\ref{equ_monotically_P}) is verified.
The eigenvalues of $\tilde{\mathcal{L}}$ are $\lambda_1 = 0$, $\lambda_2 = 1.5 + j0.5$, and $\lambda_3 = 1.5 - j0.5$.
When $\lambda_2 = 1.5 \pm j0.5$, the left-hand side of (\ref{equ_sufficient_Q}) is $[0.7045 -0.0621; -0.0621 0.7089]$, which is positive definite.
Therefore, the condition in (\ref{equ_sufficient_Q}) is satisfied.
The simulated system trajectories are plotted in Fig. \ref{fig_2d_state1} and Fig. \ref{fig_2d_state2}, respectively.
From these two figures, it can be seen that the
consensus is reached.

\begin{figure}[!hbt] \centering
\includegraphics[width=0.5\textwidth]{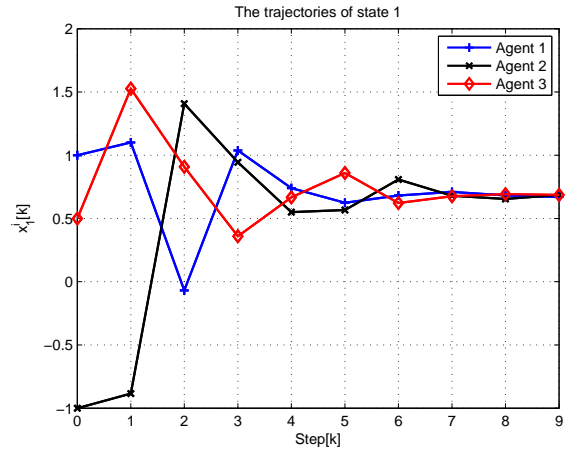}
\caption{Trajectories of state 1.}\label{fig_2d_state1}
\end{figure}
\begin{figure}[!hbt] \centering
\includegraphics[width=0.5\textwidth]{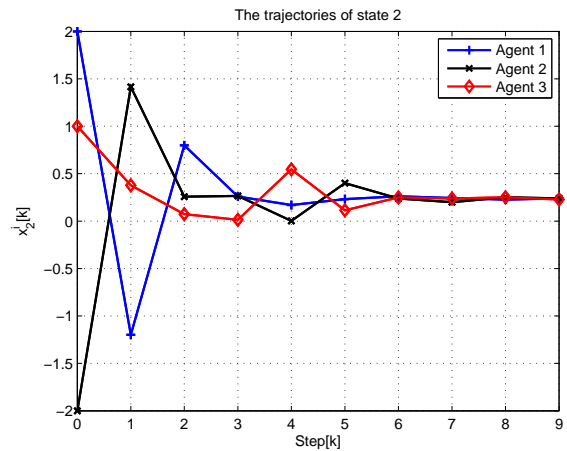}
\caption{Trajectories of state 2.}\label{fig_2d_state2}
\end{figure}

\section{CONCLUSIONS}\label{Sec_Conclusion}
In this paper, we have proposed a novel consensus scheme by using the distributed RHC for general LTI multi-agent systems.
The necessary and sufficient conditions for ensuring consensus have been developed.
Furthermore, we have developed more easily solvable conditions for multi-agent systems of general LTI and one-dimensional system dynamics, respectively. The developed theoretical results have been verified by two numerical studies.
%
%
%

\bibliographystyle{IEEEtran}
\bibliography{ref}

\begin{thebibliography}{10}
\providecommand{\url}[1]{#1}
\csname url@samestyle\endcsname
\providecommand{\newblock}{\relax}
\providecommand{\bibinfo}[2]{#2}
\providecommand{\BIBentrySTDinterwordspacing}{\spaceskip=0pt\relax}
\providecommand{\BIBentryALTinterwordstretchfactor}{4}
\providecommand{\BIBentryALTinterwordspacing}{\spaceskip=\fontdimen2\font plus
\BIBentryALTinterwordstretchfactor\fontdimen3\font minus
  \fontdimen4\font\relax}
\providecommand{\BIBforeignlanguage}[2]{{%
\expandafter\ifx\csname l@#1\endcsname\relax
\typeout{** WARNING: IEEEtran.bst: No hyphenation pattern has been}%
\typeout{** loaded for the language `#1'. Using the pattern for}%
\typeout{** the default language instead.}%
\else
\language=\csname l@#1\endcsname
\fi
#2}}
\providecommand{\BIBdecl}{\relax}
\BIBdecl

\bibitem{Reza_04_consensus_TAC}
Olfati-Saber and R.~M. Murray, ``{Consensus problems in networks of agents with
  switching topology and time-delays},'' \emph{IEEE Transactions on Automatic
  Control}, vol.~49, pp. 1520--1533, 2004.

\bibitem{Moreau_05_TAC_consensus}
L.~Moreau, ``{Stability of multiagent systems with time-dependent communication
  links},'' \emph{IEEE Transactions on Automatic Control}, vol.~50, pp.
  169--182, 2005.

\bibitem{Wei_05_TAC_consensus}
W.~Ren and R.~W. Beard, ``{Consensus seeking in multiagent systems under
  dynamically changing interaction topologies},'' \emph{IEEE Transactions on
  Automatic Control}, vol.~50, pp. 655--661, 2005.

\bibitem{Cao_08_TSMC_consensus}
Y.~Cao and W.~Ren, ``Optimal linear-consensus algorithms: An {LQR}
  perspective,'' \emph{IEEE Transactions on Systems, Man, and Cybernetics, Part
  B: Cybernetics}, vol.~40, no.~3, pp. 819--830, 2010.

\bibitem{Movric_14_TAC_optimal_tracking}
K.~Movric and F.~Lewis, ``Cooperative optimal control for multi-agent systems
  on directed graph topologies,'' \emph{IEEE Transactions on Automatic
  Control}, vol.~59, no.~3, pp. 769--774, 2014.

\bibitem{Borrelli_08_TAC_distributedLQR}
F.~Borrelli and T.~Keviczky, ``Distributed {LQR} design for identical
  dynamically decoupled systems,'' \emph{IEEE Transactions on Automatic
  Control}, vol.~53, no.~8, pp. 1901--1912, 2008.

\bibitem{Dunbar_06_Auto_DMPC}
W.~B. Dunbar and R.~M. Murray, ``Distributed receding horizon control for
  multi-vehicle formation stabilization,'' \emph{Automatica}, vol.~42, no.~4,
  pp. 549--558, 2006.

\bibitem{Muller_12_DMPC_IJRNC}
M.~A. M\"{u}ller, M.~Reble, and F.~Allg\"{o}wer, ``Cooperative control of
  dynamically decoupled systems via distributed model predictive control,''
  \emph{International Journal of Robust and Nonlinear Control}, vol.~22,
  no.~12, pp. 1376--1397, 2012.

\bibitem{Keviczky_06_auto_DMPC}
T.~Keviczky, F.~Borrelli, and G.~J. Balas, ``Decentralized receding horizon
  control for large scale dynamically decoupled systems,'' \emph{Automatica},
  vol.~42, no.~12, pp. 2105--2115, 2006.

\bibitem{Keviczky_08_TCST_formation_MPC}
T.~Keviczky, F.~Borrelli, K.~Fregene, D.~Godbole, and G.~J. Balas,
  ``Decentralized receding horizon control and coordination of autonomous
  vehicle formations,'' \emph{IEEE Transactions on Control Systems Technology},
  vol.~16, no.~1, pp. 19--33, 2008.

\bibitem{Richards_IJC_07_RDMPC}
A.~Richards and J.~P. How, ``Robust distributed model predictive control,''
  \emph{International Journal of Control}, vol.~80, no.~9, pp. 1517--1531,
  2007.

\bibitem{Li_TAC_14_RDRHC}
H.~Li and Y.~Shi, ``Robust distributed model predictive control of constrained
  continuous-time nonlinear systems: A robustness constraint approach,''
  \emph{IEEE Transactions on Automatic Control}, vol.~59, no.~6, pp.
  1673--1678, 2014.

\bibitem{Li_SCL_DMPC_14}
------, ``Distributed model predictive control of constrained nonlinear systems
  with communication delays,'' \emph{Systems \& Control Letters}, vol.~62,
  no.~10, pp. 819 -- 826, 2013.

\bibitem{Franco_08_TAC_DMPC}
E.~Franco, L.~Magni, T.~Parisini, M.~M. Polycarpou, and D.~M. Raimondo,
  ``Cooperative constrained control of distributed agents with nonlinear
  dynamics and delayed information exchange: {A stabilizing receding-horizon
  approach},'' \emph{IEEE Transactions on Automatic Control}, vol.~53, no.~1,
  pp. 324--338, 2008.

\bibitem{Li_Auto_14_DRHC}
H.~Li and Y.~Shi, ``Distributed receding horizon control of large-scale
  nonlinear systems: Handling communication delays and disturbances,''
  \emph{Automatica}, vol.~50, no.~4, pp. 1264--1271, 2014.

\bibitem{Ferrari_09_TAC_MPC_consensus}
G.~Ferrari-Trecate, L.~Galbusera, M.~P.~E. Marciandi, and R.~Scattolini,
  ``Model predictive control schemes for consensus in multi-agent systems with
  single- and double-integrator dynamics,'' \emph{IEEE Transactions on
  Automatic Control}, vol.~54, no.~11, pp. 2560--2572, 2009.

\bibitem{Zhan_13_auto_consensus}
J.~Zhan and X.~Li, ``Consensus of sampled-data multi-agent networking systems
  via model predictive control,'' \emph{Automatica}, vol.~49, no.~8, pp. 2502
  -- 2507, 2013.

\bibitem{Johansson_08_auto_optimal_consensus}
B.~Johansson, A.~Speranzon, M.~Johansson, and K.~H. Johansson, ``On
  decentralized negotiation of optimal consensus,'' \emph{Automatica}, vol.~44,
  no.~4, pp. 1175--1179, 2008.

\bibitem{You_11_TAC_consensus}
Y.~K. and X.~L., ``Network topology and communication data rate for
  consensusability of discrete-time multi-agent systems,'' \emph{IEEE
  Transactions on Automatic Control}, vol.~56, no.~10, pp. 2262--2275, 2011.

\bibitem{Lu_12_optimization}
A.~Antoniou and W.-S. Lu, \emph{Practical optimization: algorithm and
  engineering applications}.\hskip 1em plus 0.5em minus 0.4em\relax Springer,
  2007.

\bibitem{Poubelle_88_TAC_FakeReccati}
M.~G. M~A~Poubelle, R R~Bitmead, ``Fake algebraic riccati techniques and
  stability,'' \emph{IEEE Transactions on Automatic Control}, vol.~33, no.~4,
  pp. 379--381, 1988.

\bibitem{Anderson_71_LOC}
J.~B.~M. B.~O.~Anderson, \emph{Linear optimal control}.\hskip 1em plus 0.5em
  minus 0.4em\relax Prentice Hall, 1971.

\end{thebibliography}

\end{document}